\documentclass[11pt]{article}
\usepackage{amscd}
\usepackage{amsfonts}
\usepackage{amsmath}
\usepackage{amssymb}
\usepackage{amsthm}
\usepackage{bbm}
\usepackage{CJK}
\usepackage{fancyhdr}
\usepackage{graphicx}
\usepackage{hyperref}
\usepackage{indentfirst}
\usepackage{latexsym}
\usepackage{mathrsfs}
\usepackage{xypic}
\usepackage{tikz-cd}

\newtheorem{theorem}{Theorem}[section]
\newtheorem{lemma}[theorem]{Lemma}
\newtheorem{definition}[theorem]{Definition}
\newtheorem{proposition}[theorem]{Proposition}
\newtheorem{example}[theorem]{Example}

\newtheorem{remark}[theorem]{Remark}

\usepackage[top=1in,bottom=1in,left=1.25in,right=1.25in]{geometry}
\def\<{\langle}
\def\>{\rangle}

\date{}
\begin{document}
\renewcommand{\baselinestretch}{1.2}
\renewcommand{\arraystretch}{1.0}
\title{\bf Cohomology of modified Rota-Baxter Leibniz algebra of weight $\lambda$}
\author{{\bf Bibhash Mondal$^{1}$,     Ripan Saha$^{2}$\footnote
        { Corresponding author (Ripan Saha),  Email: ripanjumaths@gmail.com}}\\
%{\small 1. School of Mathematics and Statistics, Guizhou University of Finance and Economics} \\
%{\small  Guiyang  550025, P. R. of China}\\
  {\small 1. Department of Mathematics, Behala College}\\
  {\small Behala, 700060, Kolkata, India}\\  
  %{\small Email: mondaliiser@gmail.com}\\
 {\small 2. Department of Mathematics, Raiganj University} \\
{\small  Raiganj 733134, West Bengal, India}}
 \maketitle
\begin{center}
\begin{minipage}{13.cm}

{\bf \begin{center} ABSTRACT \end{center}}
Rota-Baxter operators have been paid much attention in the last few decades as they have many applications in mathematics and physics. In this paper, our object of study is modified Rota-Baxter operators on Leibniz algebras. We investigate modified Rota-Baxter Leibniz algebras from the cohomological point of view. We study a one-parameter formal deformation theory of modified Rota-Baxter Leibniz algebras and define the associated deformation cohomology that controls the deformation. Finally, as an application, we characterize equivalence classes of abelian extensions in terms of second cohomology groups.

 \smallskip

 \smallskip

 \end{minipage}
 \end{center}
 \normalsize\vskip0.5cm

\section{Introduction}

 \medskip
Rota-Baxter operator first appeared in fluctuation theory in Probability  \cite{Baxter}. It was further developed by  Rota \cite{Rota} and Cartier \cite{Cartier}. The Rota-Baxter operator has been studied on various algebraic structures like associative algebras \cite{das20}, Lie algebras \cite{guo12, GLS, Das3}, Pre-Lie algebras \cite{LHB}, Leibniz algebras \cite{MoS22}, etc. Leibniz algebras are often considered noncommutative generalizations of Lie algebras. Many important results about Lie algebras have been extended to Leibniz algebras.\\
A linear operator $T$ on  a Leibniz algebra $(\mathfrak{g},[~,~])$ is called  a Rota-Baxter operator of weight $\lambda\in \mathbb{K}$, where $\mathbb{K}$ being the ground field, if $[Ta,Tb]=T([a,Tb]+[Ta,b]) + \lambda T[a,b] $, for all $a,b \in \mathfrak{g}.$ For a given Leibniz algebra $(\mathfrak{g},[~,~])$, a modified Rota-Baxter operator of weight $\lambda$ on $(\mathfrak{g},[~,~])$ is a linear operator $T: \mathfrak{g}\rightarrow \mathfrak{g}$, such that $[Ta,Tb]=T([a,Tb]+[Ta,b])+\lambda [a,b]$ for all $a,b \in \mathfrak{g}.$
 For associative algebras, a close relationship between the above two operators is proved in \cite{Zhang1, Zhang2}. Tang, Sheng, and Zhou \cite{Tang} studied the deformation theory of the relative Rota-Baxter operator on Leibniz algebras. Mishra, Das, and Hazra \cite {Das3} studied non-abelian extensions of Rota-Baxter Lie algebras. The cohomology and abelian extensions of relative Rota-Baxter Lie algebras studied in \cite{Jiang}. The Rota-Baxter operator of arbitrary weight $\lambda$ on pre-Lie algebras studied in \cite{Guo}, the modified Rota-Baxter operator on an associative algebra studied in \cite{Das2}.\\
 
In the 1960s, parallel to the Analytical Deformation Theory, Murray Gerstenhaber \cite{G63, G64} studied the formal deformation theory of associative algebras. One needs a suitable cohomology, called deformation cohomology, which controls deformations in question to study the deformation theory of a type of algebras. Gerstenhaber showed that Hochschild cohomology controls one-parameter formal deformation of associative algebras. Nijenhuis and Richardson studied formal deformation theory for Lie algebras \cite{NR66}. There is a vast literature on this subject; for some related work, see \cite{fox, HLS, MS19, MS22, MoS22, saha}.
 
  This paper studies the modified Rota-Baxter operator of weight $\lambda$ on a Leibniz algebra. First we define the notion of  modified Rota-Baxter operator $T$ of weight $\lambda$ on a Leibniz algebra $(\mathfrak{g},[~,~])$. Now using this Rota-Baxter operator $T$, we define a new bilinear bracket $[a,b]_T=[a, Tb]+[Ta,b]$ for all $a,b \in \mathfrak{g}.$ Hence, we get a new Leibniz algebra $(\mathfrak{g},[~,~]_T)$ with a new representation $(V,l_V^{'},r_V^{'})$ induced by $T$. We use the Loday-Pirashvili {\cite{Loday}} cohomology to define the cohomology of modified Rota-Baxter Leibniz algebra of weight $\lambda.$ Next, we use our cohomology to study the formal deformation theory of the modified  Rota-Baxter Leibniz algebra of weight $\lambda$. Finally, we define an abelian extension of the modified Rota-Baxter operator of weight $\lambda$ and prove how equivalent classes of abelian extensions are related to cohomology groups.
 
 The paper is organized in the following way: In Section \ref{sec2}, we recall some basic definitions of Leibniz algebra, a representation of Leibniz algebra, and define the notion of modified Rota-Baxter of weight $\lambda$ on a Leibniz algebra and its representation. In Section \ref{sec3}, we get a new modified Rota-Baxter Leibniz algebra of weight $\lambda$ and a representation of it induced by a modified Rota-Baxter of weight $\lambda$ defined on a Leibniz algebra. In Section \ref{sec4}, we define the cohomology of the modified Rota-Baxter Leibniz algebra of weight $\lambda$ using Loday-Pirashvili cohomology of the induced Leibniz algebra. In Section \ref{sec5}, we discuss one-parameter formal deformation of modified Rota-Baxter Leibniz algebra of weight $\lambda$, and show that our cohomology is the deformation cohomology. Finally, in Section \ref{sec6}, we discuss an abelian extension of the modified Rota-Baxter operator of weight $\lambda$ and characterize extensions in terms of our second cohomology groups.

 Throughout this paper, all the vector spaces are over the field $\mathbb{K}$ of characteristic zero.

\section{Preliminaries}\label{sec2}
\def\theequation{\arabic{section}.\arabic{equation}}
\setcounter{equation} {0}
\def\theequation{\arabic{section}. \arabic{equation}}
\setcounter{equation} {0}
In this section, we recall some definitions and define the notion of the modified Rota-Baxter operator of weight $\lambda$ on a Leibniz algebra and its representation.
\begin{definition}[\cite{Loday}]
A Leibniz algebra is a vector space $\mathfrak{g}$ along with a bilinear product called bracket defined by  $[~,~] : \mathfrak{g} \times \mathfrak{g} \rightarrow  \mathfrak{g}$ such that the following condition(Leibniz identity) is satisfied 
\[[a,[b,c]]=[[a,b],c]+[b,[a,c]]~~~ \mbox{for all }~ a,b,c \in \mathfrak{g}.\] 
A Leibniz algebra $\mathfrak{g}$ with bracket $[~,~]$ is denoted by $(\mathfrak{g},[~,~]).$ Note that the above definition is the definition of left Leibniz algebra. In this paper, we consider left Leibniz algebra as Leibniz algebra.
\end{definition}
\begin{example}[\cite{Demir}] \label{exam}Consider the usual 2-dimensional vector space $\mathbb{R}^2$ with standard basis $\{e_1,e_2\}$. Now define a bilinear bracket by $[e_2,e_1]=[e_2,e_2]=e_1$, and all other brackets on basis elements being zero. Then $(\mathbb{R}^2,[~,~])$ is a Leibniz algebra.
\end{example}

\begin{definition}[\cite{Loday}]
Let $(\mathfrak{g},[~,~])$ be a Leibniz algebra and $V$ be a vector space. Suppose two maps  defined by $l_V : \mathfrak{g} \otimes V \rightarrow V$(called left action) and $r_V : V \otimes \mathfrak{g} \rightarrow V$(called right action) so that the following conditions are satisfied
\begin{align*}
l_V(a,l_V(b,v))=l_V([a,b],v)+l_V(b,l_V(a,v))\\
l_V(a,r_V(v,b))=r_V(l_V(a,v),b)+r_V(v,[a,b])\\
 r_V(v,[a,b])= r_V(r_V(v,a),b)+l_V(a,r_V(v,b)),
 \end{align*}
for all $a,b \in \mathfrak{g}$ and $v\in V.$ Then $(V,l_V,r_V)$ is called a \textbf {representation} of the Liebniz algebra $(\mathfrak{g},[~,~]).$
\end{definition} 

\begin{definition}[\cite{Tang}]
Let $(\mathfrak{g},[~,~])$ be a Leibniz algebra. A linear operator $T$ on $\mathfrak{g}$ is called  a \textbf{Rota-Baxter operator} on $(\mathfrak{g},[~,~])$ if $[Ta,Tb]=T([a,Tb]+[Ta,b])$ for all $a,b \in \mathfrak{g}.$
\end{definition}
\begin{definition}
Let $(\mathfrak{g},[~,~])$ be a Leibniz algebra. A linear operator $T$ on $\mathfrak{g}$ is called  a \textbf{Rota-Baxter operator of weight $\lambda \in \mathbb{K}$} on $(\mathfrak{g},[~,~])$ if $[Ta,Tb]=T([a,Tb]+[Ta,b]) + \lambda T[a,b]$, for all $a,b \in \mathfrak{g}.$
\end{definition}
\begin{definition}
Let $(\mathfrak{g},[~,~])$ be a Leibniz algebra. A \textbf{modified Rota-Baxter operator of weight $\lambda$} on $(\mathfrak{g},[~,~])$ is a linear operator $T: \mathfrak{g}\rightarrow \mathfrak{g}$, such that $[Ta,Tb]=T([a,Tb]+[Ta,b])+\lambda [a,b]$ for all $a,b \in \mathfrak{g}.$
\end{definition}
 
\begin{example}
Consider the Leibniz algebra $(\mathbb{R}^2,[~,~])$ defined in Example \ref{exam}. Then the linear map $T: \mathbb{R}^2 \rightarrow \mathbb{R}^2,x \mapsto Ax$ where $A= \begin{bmatrix}
1 &4 \\
0 & -3 \\
\end{bmatrix}$ is a modified Rota-Baxter operator of weight $-1$ on $(\mathbb{R}^2,[~,~]).$
\end{example}

Note that a modified Rota-Baxter operator of weight $0$ on a Leibniz algebra $(\mathfrak{g},[~,~])$  is the same as the Rota-Baxter operator on that Leibniz algebra.

\begin{definition}
A \textbf{modified Rota-Baxter Leibniz algebra of weight $\lambda$}  is a Leibniz algebra $(\mathfrak{g},[~,~])$ equipped with a modified Rota-Baxter operator $T : \mathfrak{g} \rightarrow \mathfrak{g}$ of weight $\lambda.$ We denote it by $(\mathfrak{g},[~,~],T).$
\end{definition}

\begin{remark}
Semenov-Tyan-Shanskii\cite{sem} observed intricate relationship between Rota-Baxter operator and the classical Yang-Baxter equation (CYBE) under specific conditions. Moreover, within the same discourse, a concept termed the modified classical Yang-Baxter equation (modified CYBE) is introduced, boasting solutions known as modified r-matrices. The modified CYBE, born from a transformation of the CYBE, has proven to be instrumental in multifaceted mathematical domains. Remarkably, despite its genesis from the CYBE, the modified CYBE plays an independent and pivotal role in mathematical physics, attracting various authors to examine its implications autonomously. Inspired by this line of inquiry, researchers in \cite{Zhang1, Zhang2} explored the associative counterpart of the modified CYBE, known as the modified associative Yang-Baxter equation of weight $\lambda$ (abbreviated as modified AYBE of weight $\lambda \in \mathbb{K}$). A solution to the modified AYBE of weight $\lambda$ is termed a modified Rota-Baxter operator of weight $\lambda$.

Rota-Baxter operator and modified Rota-Baxter operator are nicely related with each other \cite{FG}. We state the following results expressing their relationship and skip the routine proofs.
\begin{itemize}
\item[a)] Let $(\mathfrak{g}, T)$ be a Rota-Baxter operator of weight $\lambda$. Define $S:= -\lambda Id_{\mathfrak{g}} - 2T$. Then $(\mathfrak{g}, S)$ is a modified Rota-Baxter operator and vice-versa. Observe that a Rota-Baxter operator of weight zero is a modified Rota-Baxter operator of weight zero.
\item[b)] Let $\mathfrak{g}$ be a $\mathbb{K}$-algebra. Then $T: \mathfrak{g} \to \mathfrak{g}$ is modified Rota-Baxter operator if and only if $-T$ is modified Rota-Baxter operator.
\end{itemize}
There is a  close relationship between associative algebras and Lie algebras, and Leibniz algebras are often considered non-commutative generalizations of Lie algebras. In \cite{MoS22, Tang}, the authors studied Rota-Baxter operators on Leibniz algebras. This exploration naturally prompts an investigation into modified Rota-Baxter operators on Leibniz algebras.
\end{remark}
\begin{definition}
Let $(\mathfrak{g},[~,~],T)$ and $(\mathfrak{g^{'}},[~,~]^{'},{T^{'}})$ be two modified  Rota-Baxter of Leibniz algebra of weight $\lambda.$ Then a map $\phi : \mathfrak{g} \rightarrow \mathfrak{g^{'}}$ is called a \textbf{morphism of modified Rota-Baxter Leibniz algebra of weight $\lambda$} if the map $\phi$ is a Leibniz algebra morphism and $T^{'} \circ \phi = \phi \circ T.$
\end{definition}
\begin{definition}
Let $(\mathfrak{g},[~,~],T)$ be a modified Rota-Baxter Leibniz algebra of weight $\lambda$ and $V$ be a vector space. We define a representation of modified Rota-Baxter Leibniz algebra of weight $\lambda$ as a quadruple $(V,l_V,r_V,T_V)$, where $(V,l_V,r_V)$ is a representation of the Leibniz algebra $(\mathfrak{g},[~,~])$ and $T_V$ is a linear map on $V$ satisfies the following conditions  
\[l_V(Ta,T_V(v))=T_V(l_V(Ta,v)+l_V(a,T_V(v)))+\lambda l_V(a,v)\]
\[r_V(T_V(v),Ta)=T_V(r_V(T_V(v),a)+r_V(v,Ta))+\lambda r_V(v,a)\]
for all $a \in \mathfrak{g}$ and $v \in V.$
\end{definition}
Let $(\mathfrak{g},[~,~],T)$ be a modified Rota-Baxter Leibniz algebra of weight $\lambda.$ Now define $l_{\mathfrak{g}}: \mathfrak{g} \times \mathfrak{g} \mapsto \mathfrak{g}$ and $r_{\mathfrak{g}}: \mathfrak{g} \times \mathfrak{g} \mapsto \mathfrak{g}$ by $l_{\mathfrak{g}}(a,b)=r_{\mathfrak{g}}(a,b)=[a,b]$ for all $a,b \in \mathfrak{g}$. Then $(\mathfrak{g},l_{\mathfrak{g}},r_{\mathfrak{g}},T)$ is a representation of the modified Rota-Baxter Leibniz algebra $(\mathfrak{g},[~,~],T)$ of weight $\lambda.$ We call this representation as an adjoint representation of $(\mathfrak{g},[~,~],T)$.
\section{Induced modified Rota-Baxter Leibniz algebra of weight $\lambda$ and its representation}\label{sec3}
In this section, we study induced modified Rota-Baxter Leibniz algebra of weight $\lambda$ from a modified Rota-Baxter operator of weight $\lambda$.
\begin{theorem}
Let $(\mathfrak{g},[~,~],T)$ be a modified Rota-Baxter Leibniz algebra of weight $\lambda$. Now we define a bracket $[a,b]_T=[a,Tb]+[Ta,b]$ for all $a,b \in \mathfrak{g}.$ We have the following results 
\begin{itemize}
\item[(a)] $(\mathfrak{g},[~,~]_T)$ is a Leibniz algebra.
\item[(b)] $T$ is a modified Rota-Baxter operator of weight $\lambda$ on $(\mathfrak{g},[~,~]_T).$ 
\end{itemize}
\end{theorem}
 \begin{proof}
 \begin{itemize}
 
 \item[(a)]
 Clearly, $[~,~]_T$ is a bilinear map.
 Now, for any $a,b,c \in  \mathfrak{g}$, we have
 \begin{align*}
 &[a,[b,c]_T]_T \\
 &=[a,[Tb,c]+[b,Tc]]_T\\
 &=[a,T([Tb,c]+[b,Tc])]+[Ta,[Tb,c]+[b,Tc]] \\
 &=[a,[Tb,Tc]-\lambda [a,[b,c]]+[Ta,[Tb,c]]+[Ta,[b,Tc]]\\
 &=[a,[Tb,Tc]]-\lambda [a,[b,c]]+[Ta,[Tb,c]]+[Ta,[b,Tc]].\\
 \end{align*}
 Similarly, we have 
 \[[[a,b]_T,c]_T=[[Ta,Tb],c]-\lambda [[a,b],c]+[[Ta,b],Tc]+[[a,Tb],Tc],\] and
 \[[b,[a,c]_T]_T=[b,[Ta,Tc]]-\lambda[b,[a,c]]+[Tb,[Ta,c]]+[Tb,[a,Tc]].\]
 
 Now using the Leibniz identity of $(\mathfrak{g},[~,~])$ we have $[a,[b,c]_T]_T=[[a,b]_T,c]_T+[b,[a,c]_T]_T$ for any $a,b,c \in \mathfrak{g}.$ Hence, $(\mathfrak{g},[~,~]_T)$ is  Leibniz algebra.
 \item[(b)] Again, for any $a,b \in \mathfrak{g}$, we have
 \begin{align*} 
 &[Ta,Tb]_{T}\\
&= [Ta,T(Tb)]+[T(Ta),Tb]\\
&= T([Ta,Tb]+[a,T(Tb)])+\lambda [a,Tb]+T([T(Ta),b]+[Ta,Tb])+\lambda [Ta,b]\\
&= T([Ta,b]_T+[a,Tb]_T)+\lambda [a,b]_T.
\end{align*}
Hence, $T$ is a modified Rota-Baxter operator of weight $\lambda$ on $(\mathfrak{g},[~,~]_T).$
 \end{itemize}
 \end{proof}
\begin{theorem}
Let $(\mathfrak{g},[~,~],T)$ be a modified Rota-Baxter Leibniz algebra of weight $\lambda$ and $(V,l_V,r_V,T_V)$ be a representation of it. Now define two maps $l_V^{'}: \mathfrak{g} \otimes V \rightarrow V$ and $r_V^{'}: V \otimes \mathfrak{g} \rightarrow V$ as follows 
\[l_V^{'}(a,v)=l_V(Ta,v)-(T_V\circ l_V)(a,v)~~ \mbox{and}~~ r_V^{'}(v,a)=r_V(v,Ta)-(T_V\circ r_V)(v,a),\]
for all $a \in \mathfrak{g}$ and $v \in V.$ Then  $(T,l_V^{'},r_V^{'},T_V)$ is a representation of the modified Rota-Baxter Leibniz algebra $(\mathfrak{g},[~,~]_T,T)$ of weight $\lambda$.
\end{theorem}

\begin{proof}
Let $a,b \in \mathfrak{g},v \in V$. We have
\begin{align*}
&l_V^{'}(a,l_V^{'}(b,v))-l_V^{'}([a,b]_T,v)-l_V^{'}(b,l_V^{'}(a,v))\\
&=l_V(Ta,l_V^{'}(b,v))-(T_V\circ l_V)(a,l_V^{'}(b,v))\\
&-l_V(T([a,b]_T),v)+(T_V \circ l_V)([a,b]_T,v)\\
&-l_V(Tb,l_V^{'}(a,v))+(T_V\circ l_V)(b,l_V^{'}(a,v))\\
&=l_V(Ta,l_V(Tb,v))-l_V(Ta,(T_V\circ l_V)(b,v))\\
&-(T_V\circ l_V)(a,l_V(Tb,v))+(T_V\circ l_V)(a,(T_V\circ l_V) (b,v))\\
&-l_V([Ta,Tb],v)+\lambda~ l_V([a,b],v)\\
& +(T_V\circ l_V)([a,Tb],v)+(T_V\circ l_V)([Ta,b],v)\\
&-l_V(Tb,l_V(Ta,v))+l_V(Tb,(T_V\circ l_V) (a,v))\\
&+(T_V\circ l_V)(b,l_V(Ta,v))-(T_V\circ l_V)(b,(T\circ l_V)(a,v))\\
&=\bigg (l_V(Ta,l_V(Tb,v))-l_V([Ta,Tb],v)-l_V(Tb,l_V(Ta,v)) \bigg )\\
&-\bigg (l_V(Ta,(T_V\circ l_V)(b,v))-\lambda ~ l_V(a,l_V(b,v))\bigg)\\
&-(T_V\circ l_V)(a,l_V(Tb,v))+(T_V\circ l_V)(a,(T_V\circ l_V) (b,v))\\
&+(T_V\circ l_V)([a,Tb],v)+(T_V\circ l_V)([Ta,b],v)\\
&+\bigg (l_V(Tb,(T_V\circ l_V) (a,v))- \lambda ~ l_V(b,l_V(a,v))\bigg)+(T_V\circ l_V)(b,l_V(Ta,v))\\
&-(T_V\circ l_V)(b,(T_V\circ l_V)(a,v))\\
&=-T_V\bigg(l_V(Ta,l_V(b,v))+l_V(a,(T_V\circ l_V)(b,v))\bigg)-(T_V\circ l_V)(a,l_V(Tb,v))\\
&+(T_V\circ l_V)(a,(T_V\circ l_V) (b,v))+(T_V\circ l_V)([a,Tb],v)+(T_V\circ l_V)([Ta,b],v)\\
&+T_V\bigg(l_V(Tb,l_V(a,v))+l_V(b,(T_V\circ l_V)(a,v))\bigg)\\
&+(T_V\circ l_V)(b,l_V(Ta,v))-(T_V\circ l_V)(b,(T_V\circ l_V)(a,v))\\
&=\bigg (-(T_V \circ l_V)(Ta,l_V(b,v))+(T_V\circ l_V)([Ta,b],v)+(T_V\circ l_V)(b,l_V(Ta,v))\bigg )\\
&\bigg (-(T_V\circ l_V)(a,l_V(Tb,v))+(T_V\circ l_V)([a,Tb],v)+(T_V \circ l_V)(Tb,l_V(a,v)) \bigg )\\
&+\bigg ( -(T_V \circ l_V)(a,(T_V\circ l_V)(b,v))+ (T_V \circ l_V)(a,(T_V\circ l_V)(b,v))\bigg )\\
&+\bigg ((T_V \circ l_V)(b,(T_V \circ l_V)(a,v))-(T_V \circ l_V)(b,(T_V \circ l_V)(a,v))\bigg )\\
&=0.\\
\end{align*}

Therefore,
\[l_V^{'}(a,l_V^{'}(b,v))=l_V^{'}([a,b]_{T},v)+l_V^{'}(b,l_V^{'}(a,v))\]
holds.
Similarly, we have the following equations
\[l_V^{'}(a,r_V^{'}(v,b))=r_V^{'}(l_V^{'}(a,v),b)+r_V^{'}(v,[a,b]_{T}),\]
\[ r_V^{'}(v,[a,b]_{T})= r_V^{'}(r_V^{'}(v,a),b)+l_V^{'}(a,r_V^{'}(v,b))\]
holds for all $a,b \in \mathfrak{g}, v \in V.$
Hence,  $(V,l_V^{'},r_V^{'} )$ is a representation of the Leibniz algebra $(\mathfrak{g},[~,~]_T,T)$.
Also, we have
\begin{align*}
&T_V(l_V^{'}(Ta,v)+l_V^{'}(a,T_V(v)))+\lambda ~ l_V^{'}(a,v)\\
&=T_V\bigg(l_V(T(Ta),v)-(T_V\circ l_V)(Ta,v)+l_V(Ta,T_V(v))-(T_V\circ l_V)(a,T_V(v))\bigg)\\
&+\lambda ~(l_V(Ta,v)-(T_V \circ l_V)(a,v))\\
&=T_V\bigg(l_V(T(Ta),v)+l_V(Ta,T_V(v))\bigg)+ \lambda ~ l_V(Ta,v)-T_V\bigg( T_V(l_V(a,T_V(v))\\
&+l_V(Ta,v))+\lambda l_V(a,v)\bigg)\\
&=l_V(T(Ta),T_V(v))-T_V ( l_V(Ta,T_V(v)))\\
&=l_V(T(Ta),T_V(v))-(T_V \circ l_V)(Ta,T_V(v))\\
&=l_V^{'}(Ta,T_V(v)).\\
\end{align*}
Hence, 
\[ l_V^{'}(Ta,T_V(v))=T_V(l_V^{'}(Ta,v)+l_V^{'}(a,T_V(v)))+\lambda l_V^{'}(a,v)\]
holds for all $a \in \mathfrak{g}$ and $v \in V$. Similarly, it can be proven that 
\[r_V^{'}(T_V(v),Ta)=T_V(r_V^{'}(T_V(v),a)+r_V^{'}(v,Ta))+\lambda r_V^{'}(v,a)\]
for all $a \in \mathfrak{g}$ and $v \in V.$

Hence, $(V,l_V^{'},r_V^{'},T_V )$ is a representation of the modified Rota-Baxter Leibniz algebra $(\mathfrak{g},[~,~]_T,T)$ of weight $\lambda$.

\end{proof}
\section{Cohomology of modified Rota-Baxter Leibniz algebra of weight $\lambda$}\label{sec4}
In this section, we define the cohomology of modified Rota-Baxter Leibniz algebra of weight $\lambda$ using the Loday- Pirashvili cohomology (\cite{Loday}) for Leibniz algebras.

Let $(\mathfrak{g},[~,~],T)$ be a modified Rota-Baxter Leibniz algebra of weight $\lambda$ and $(V,l_V,r_V,T_V)$ be a representation of it. Then $(V,l_V,r_V)$ is a representation of the Leibniz algebra $(\mathfrak{g},[~,~])$. Now using the Loday- Pirashvili cohomology  (\cite{Loday}) for this Leibniz algebra, we have  abelian groups $C^n_{LA}(\mathfrak{g},V):=Hom(\mathfrak{g}^{\otimes n},V)$ and coboundary maps $\delta ^n: C^n_{LA}(\mathfrak{g},V) \rightarrow C^{n+1}_{LA}(\mathfrak{g},V)$  defined by
\begin{align*}
 &(\delta^n(f))(a_1,a_2,\ldots ,a_{n+1})\\
 & =\sum_{i=1}^{n}(-1)^{i+1}l_{V}(a_i,f(a_1,\ldots,\hat{a_i},\ldots, a_{n+1}))+(-1)^{n+1}r_{V}(f(a_1,\ldots,a_n),a_{n+1}) \\
 &+ \sum_{1 \leq i <j \leq n+1}(-1)^if(a_1,\ldots, \hat{a_i},\ldots ,a_{j-1},[a_i,a_j],a_{j+1},\ldots, a_{n+1}),
  \end{align*}
where $f\in C^n_{LA}(\mathfrak{g},V)$ and $a_1,\ldots,a_{n+1} \in \mathfrak{g}$. Now using the notation $l_V(a,v)=[a,v]$ and $r_V(v,a)=[v,a]$, the above coboundary map reduces to $\delta ^n: C^n_{LA}(\mathfrak{g},V) \rightarrow C^{n+1}_{LA}(\mathfrak{g},V)$ , where
 \begin{align*}
& (\delta^n(f))(a_1,a_2,\ldots ,a_{n+1})\\
 & =\sum_{i=1}^{n}(-1)^{i+1}[a_i,f(a_1,\ldots,\hat{a_i},\ldots, a_{n+1})]+(-1)^{n+1}[f(a_1,\ldots,a_n),a_{n+1}] \\
 &+ \sum_{1 \leq i <j \leq n+1}(-1)^if(a_1,\ldots, \hat{a_i},\ldots ,a_{j-1},[a_i,a_j],a_{j+1},\ldots, a_{n+1}),
  \end{align*}
where $f\in C^n_{LA}(\mathfrak{g},V)$ and $a_1,\ldots,a_{n+1} \in \mathfrak{g}$. 

Again, from the modified Rota-Baxter Leibniz algebra $(\mathfrak{g},[~,~],T)$ of weight $\lambda$ with representation $(V,l_V,r_V,T_V)$, we have the induced Leibniz algebra $(\mathfrak{g},[~,~]_T)$ with representation $(V,l_V^{'},r_V^{'})$. Hence, using the Loday- Pirashvili cohomology (\cite{Loday}) for this Leibniz algebra, we have the cochain groups and coboundary map as follows; for each  $n \geq 0$, we have   $C^n_{mRBO}(\mathfrak{g}, V):=Hom(\mathfrak{g}^{\otimes n}, V)$ and  coboundary map $\partial ^n: C^n_{mRBO}(\mathfrak{g}, V) \rightarrow C^{n+1}_{mRBO}(\mathfrak{g}, V)$ defined by
\begin{align*}
 &(\partial^n(f))(a_1,a_2,\ldots ,a_{n+1})\\
 &=\sum_{i=1}^{n}(-1)^{i+1}l^{'}_{V}(a_i,f(a_1,\ldots,\hat{a_i},\ldots, a_{n+1}))+(-1)^{n+1}r^{'}_{V}(f(a_1,\ldots,a_n),a_{n+1})\\
& + \sum_{1 \leq i <j \leq n+1}(-1)^if(a_1,\ldots, \hat{a_i},\ldots ,a_{j-1},[a_i,a_j]_{T},a_{j+1},\ldots, a_{n+1}),\\
\end{align*}
\begin{align*}
&=\sum _{i=1}^{n}(-1)^{i+1}[T(a_i),f(a_1,\ldots,\hat{a_i},\ldots, a_{n+1})] -\sum _{i=1}^{n}(-1)^{i+1}T_V([a_i,f(a_1,\ldots,\hat{a_i},\ldots, a_{n+1})])\\
&+(-1)^{n+1}[f (a_1,\ldots ,a_{n}),T(a_{n+1})] -(-1)^{n+1}T_V([f (a_1,\ldots ,a_{n}),a_{n+1}])\\
&+\sum_{1\leq i< j\leq n+1}(-1)^if (a_1,\ldots,\hat{a_i},\ldots,a_{j-1},[T(a_i),a_j]+[a_i,T(a_j)],a_{j+1},\ldots ,a_{n+1}),
 \end{align*}
where $f\in C^n_{mRBO}(\mathfrak{g},V)$ and $a_1,\ldots,a_{n+1} \in \mathfrak{g}$.\\
Now, motivated by the Proposition 3.2 of \cite{Das2} (similar result for associative algebra), we define
\begin{definition}
Let $(\mathfrak{g},[~,~],T )$ be a modified Rota-Baxter Leibniz algebra of weight $\lambda$  and $(V,l_V,r_V,T_V)$ be a representation of it. Now, we define a map  
\[\phi^{n} : C^{n}_{LA}(\mathfrak{g},V) \rightarrow C^{n}_{mRBO}(\mathfrak{g},V)\] 
as follows:
\begin{align*}
&\phi^0(f)=Id_V;\\
&\phi^n(f)(a_1,a_2,\ldots,a_n)=f(Ta_1,Ta_2,\ldots, Ta_n)\\
&-\sum _{1\leq i_1<i_2< \cdots < i_r\leq n ,r~ \mbox{odd}}{(-\lambda)^{\frac{r-1}{2}}} (T_V\circ f)(T(a_1),\ldots,a_{i_1},\ldots,a_{i_r},\ldots,T(a_n))\\
&-\sum _{1\leq i_1<i_2< \cdots < i_r\leq n ,r~ \mbox{even}}{(-\lambda)^{\frac{r}{2}+1}} (T_V\circ f)(T(a_1),\ldots,a_{i_1},\ldots,a_{i_r},\ldots,T(a_n)),
\end{align*}
where $n\geq 1$, $f\in C^n_{LA}(\mathfrak{g},V)$ and $a_1,\ldots,a_{n+1} \in \mathfrak{g}$.
\end{definition}
\begin{lemma}\label{lemma 4.2}
$\phi^{n+1}(\delta^n(f))(a_1,a_2,a_3,\ldots,a_{n+1})=\partial^n(\phi^n(f))(a_1,a_2,a_3,\ldots,a_{n+1})$, where $f\in C^n_{LA}(\mathfrak{g},V)$ and $a_1,\ldots,a_{n+1} \in \mathfrak{g}$. 
\end{lemma}
\begin{proof}
The proof is similar to Proposition 5.2 of \cite{Wang}.
\end{proof}
Now, using the above lemma, we get the following commutative diagram
\[\begin{tikzcd}
	{C^1_{LA}(\mathfrak{g},V)} & {C^2_{LA}(\mathfrak{g},V)} & {C^n_{LA}(\mathfrak{g},V)} & {C^{n+1}_{LA}(\mathfrak{g},V)} & {} \\
	{C^1_{mRBO}(\mathfrak{g},V)} & {C^2_{mRBO}(\mathfrak{g},V)} & {C^n_{mRBO}(\mathfrak{g},V)} & {C^{n+1}_{mRBO}(\mathfrak{g},V)} & {}
	\arrow[from=1-1, to=1-2,]{r}{\delta^1}
	\arrow[from=1-1, to=2-1]{d}{\phi^1}
	\arrow[from=1-2, to=2-2]{d}{\phi^2}
	\arrow[from=2-1, to=2-2]{r}{\partial^1}
	\arrow[dotted, no head, from=1-2, to=1-3]
	\arrow[dotted, no head, from=2-2, to=2-3]
	\arrow[from=1-3, to=1-4]{r}{\delta^n}
	\arrow[from=1-3, to=2-3]{d}{\phi^n}
	\arrow[from=2-3, to=2-4]{d}{\partial^n}
	\arrow[from=1-4, to=2-4]{r}{\phi^{n+1}}
	\arrow[dotted, no head, from=1-4, to=1-5]
	\arrow[dotted, no head, from=2-4, to=2-5].
\end{tikzcd}\]

\begin{definition}
Let $(\mathfrak{g},[~,~],T)$ be a modified Rota-Baxter Leibniz algebra of weight $\lambda$ and $(V,l_V,r_V,T_V)$ be a representation of it. Now, we define 
\[C^{0}_{mRBLA}(\mathfrak{g},V):= C^{0}_{LA}(\mathfrak{g},V) ~~ \mbox{and}~~ C^n_{mRBLA}(\mathfrak{g},V):= C^n_{LA}(\mathfrak{g},V)\oplus C_{mRBO}^{n-1}(\mathfrak{g},V), \forall n \geq 1,\] and a map
$d^n: C^n_{mRBLA}(\mathfrak{g},V) \rightarrow C^{n+1}_{mRBLA}(\mathfrak{g},V)$  by
\[d^n (f , g)=(\delta^n(f),- \partial^{n-1}(g)-\phi^n(f) )\]
for any $f \in C^n_{LA}(\mathfrak{g},V)$ and $g \in C^{n-1}_{mRBO}(\mathfrak{g},V).$
\end{definition}
\begin{theorem}
The map $d^n: C^n_{mRBLA}(\mathfrak{g},V) \rightarrow C^{n+1}_{mRBLA}(\mathfrak{g},V) $ defined above satisfies $d^{n+1}\circ d^n=0,$ hence, it is a coboundary map.
\end{theorem}
\begin{proof}
Suppose  $f\in C^n_{LA}(\mathfrak{g},V) $ and $g\in C^{n-1}_{mRBO}(\mathfrak{g},V)$, observe that 
\begin{align*}
d^{n+1}\circ d^{n}(f,g)&= d^{n+1}(\delta ^n(f),-\partial ^{n-1}(g)-\phi ^{n}(f))\\
&=( \delta ^{n+1}(\delta ^n(f)),-\partial^{n}(-\partial^{n-1}(g)
-\phi^n(f))-\phi^{n+1}(\delta ^n(f)))\\
&=(0,~\partial^n(\phi^n(f))-\phi^{n+1}(\delta^n(f)))~~~(\mbox{by Lemma \ref{lemma 4.2}})\\
&=0.
\end{align*}
\end{proof}
From the above theorem it follows that $\mathbf{\{C^n_{mRBLA}(\mathfrak{g},V),d^n\}}$ is a cochain complex. We define this cochain complex as the cochain complex of the modified Rota-Baxter Leibniz algebra of weight $\lambda$ with coefficients in $V.$ Let $Z_{mRBLA}^n(\mathfrak{g},V)$ denote the space of $n$-cocycles, and $B_{mRBLA}^n(\mathfrak{g},V)$ denote the space of $n$-coboundaries. Then it follows that $ B_{mRBLA}^n(\mathfrak{g},V) \subset Z_{mRBLA}^n(\mathfrak{g},V)$ for $n \geq 0.$ Now, the quotient groups defined by \[H_{mRBLA}^n(\mathfrak{g},V):=\frac{Z_{mRBLA}^n(\mathfrak{g},V)}{B_{mRBLA}^n(\mathfrak{g},V)}
,~~ \mbox{for} ~n \geq 0,\]
are called the cohomology groups of the modified Rota-Baxter Leibniz algebra $(\mathfrak{g},[~,~],T)$ of weight $\lambda.$

\begin{remark}
The fundamental concept within deformation theory asserts that across a field of characteristic 0, any deformation theory can be encapsulated by means of a differential graded Lie algebra. To elaborate, when presented with an underlying ``space" (such as a manifold or chain complex) along with a particular type of structure, there ought to exist a differential graded Lie algebra such that the structures of that type on that space are in one-to-one correspondence with the Maurer–Cartan elements of that differential graded Lie algebra.

Using Gerstenhaber bracket on Leibniz algebras (cf. Section 2 of \cite{MS23}) and ideas from the papers \cite{das20, uch}, one can construct an explicit graded Lie algebra whose Maurer-Cartan elements are precisely modified Rota-Baxter operators on Leibniz algebras. This necessitates several additional pages of effort, but we will not be implementing this in our current paper. Consequently, we refrain from delving into this direction and will not elaborate on these details here.
\end{remark}

\section{One-parameter formal  deformation of a modified Rota-Baxter Leibniz algebra of weight $\lambda$}\label{sec5}
In this section, we study a one-parameter formal deformation of modified Rota-Baxter Leibniz algebra of weight $\lambda.$ We use the notation $\mu$ for the bilinear product $[~,~]$ and the adjoint representation for modified Rota-Baxter Leibniz algebra of weight $\lambda.$
\begin{definition}
Let $(\mathfrak{g},\mu,T)$ be a modified Rota-Baxter Leibniz algebra of weight $\lambda$. A one-parameter formal deformation of $(\mathfrak{g},\mu,T)$ is a pair of power series $(\mu_t,T_t),$
\[ \mu_t=\sum _{i=0}^{\infty}\mu_it^i, ~ \mu_{i} \in C^2_{LA}(\mathfrak{g},\mathfrak{g}),~~~~ T_t= \sum_{i=0}^{\infty} T_it^i,~ T_i \in C^1_{mRBO}(\mathfrak{g},\mathfrak{g}),\]
such that $(\mathfrak{g}[[t]],\mu_t,{T_t})$ is a modified Rota-Baxter Leibniz algebra of weight $\lambda$, where  $(\mu_0,T_{0})=(\mu ,T).$
\end{definition}
Therefore, $(\mu_t, T_t)$ will be a formal one-parameter deformation of a modified Rota-Baxter Leibniz algebra $(\mathfrak{g},\mu,T)$ of weight $\lambda$ if and only if the following conditions are satisfied 
\[\mu_t(a,\mu_t(b,c))=\mu_t(\mu_t(a,b),c)+\mu_t(b,\mu_t(a,c)),\]
and \[ \mu_t(T_t(a),T_t(b))=T_t(\mu_t(a,T_t(b))+\mu_t(T_t(a),b))+ \lambda~ \mu_t (a,b),\]
for any $a,b,c \in \mathfrak{g}$.
Expanding the above equations and equating the coefficients of $t^n$($n$ non-negative integer) from both sides, we get 
\begin{align}
\sum _{\substack{i+j=n \\i,j\geq 0}} \mu_i(a,\mu_j (b,c))=\sum _{\substack{i+j=n \\i,j\geq 0}}\mu_i(\mu _j(a,b),c)+\sum _{\substack{i+j=n \\i,j\geq 0}} \mu_i(b, \mu_j (a,c)),
\end{align}
and 
\begin{align}
\sum_{\substack{i+j+k=n \\ i,j,k \geq 0}} \mu_i(T_j(a),T_k(b))=\sum_{\substack{i+j+k=n \\ i,j,k \geq 0}}T_i(\mu_j(T_k(a),b))+\sum_{\substack{i+j+k=n \\ i,j,k \geq 0}}T_i(\mu_j(a,T_k(b)))+ \lambda \mu_n(a,b).
\end{align}
Note that for $n=0$, the above equation is precisely the Leibniz identity of $(\mathfrak{g},\mu)$ and the condition for modified Rota-Baxter operator of weight $\lambda$ respectively.
\\ 
Now, putting $n=1$ in the above equations, we get 
\[ \mu (a,\mu_1(b,c))+\mu_1(a,\mu (b,c))= \mu (\mu_1(a,b),c)+\mu_1(\mu (a,b),c) \\
+\mu_1 (b, \mu (a,c))+\mu (b, \mu_1(a,c)),\]
and 
\begin{align*}
&\mu_1(T(a_1),T(a_2))+\mu (T_1(a_1),T(a_2))+\mu (T(a_1),T_1(a_2))-T_1(\mu (T(a_1),a_2))-T(\mu (T_1(a_1),a_2)) \\
&-T(\mu _1 (T(a_1),a_2))-T_1(\mu (a_1,T(a_2)))-T(\mu _1 (a_1,T(a_2)))-T(\mu (a_1,T_1(a_2)))-\lambda ~\mu_1 (a_1,a_2)=0,
\end{align*}
where $a,b,c \in \mathfrak{g}.$ From the first equation we get $\delta^2(\mu_1)(a,b,c)=0$ and from the second equation, we have 
\begin{align*}
-\partial ^1(T_1)(a_1,a_2)&=  -T(\mu_1 (a_1,T(a_2)))-T(\mu_1(T(a_1),a_2))+\mu_1(T(a_1),T(a_2))-\lambda ~\mu_1 (a_1,a_2)\\
                  &= \phi ^2(\mu_1)(a_1,a_2).
\end{align*}
Therefore, $-\partial^(T_1)-\phi^2(\mu_1)=0$. Hence, $d^2(\mu_1,T_1)=0$.\\
 This proves $(\mu_1,T_1)$ is a $2$-cocycle in the cochain complex $\{C^{n}_{mRBLA}(\mathfrak{g},\mathfrak{g}),d^n\}.$
 Thus, from the above discussion, we have the following theorem.
 \begin{theorem}\label{infy-co}
Let $(\mu_t, T_t)$ be a one-parameter formal deformation of a modified Rota-Baxter Leibniz algebra $(\mathfrak{g},\mu,T)$ of weight $\lambda$. Then $(\mu_1,T_1)$ is a $2$-cocycle in the cochain complex $\{C^{n}_{mRBLA}(\mathfrak{g},\mathfrak{g}),d^n\}.$
\end{theorem}
\begin{definition}
The $2$-cocycle $(\mu_1,T_1)$ is called the infinitesimal of the formal one-parameter deformation $(\mu_t,T_t)$ of the modified Rota-Baxter Leibniz algebra $(\mathfrak{g},\mu,T)$ of weight $\lambda$.
\end{definition}
\begin{definition}
Let $(\mu_t, T_t)$ and $(\mu_t^{'}, T_t^{'})$ be two formal one-parameter deformations of a modified Rota-Baxter Leibniz algebra $(\mathfrak{g},\mu,T )$. A formal isomorphism between these two deformations is a power series $\psi _t=\sum _{i=0}^{\infty}\psi_i t^i: \mathfrak{g}[[t]] \rightarrow \mathfrak{g}[[t]]$, where $\psi_i: \mathfrak{g} \rightarrow \mathfrak{g}$ are linear maps and  $\psi_0=Id_{\mathfrak{g}}$ such that the following conditions are satisfied
\begin{align}
&\psi_t \circ \mu^{'}_t=\mu_t \circ (\psi_t \otimes \psi_t),\\
& \psi_t \circ T_t^{'}=T_{t} \circ \psi_t.  
\end{align}
\end{definition}
Now expanding the equations (5.3) and (5.4) and equating the coefficients of $t^n$ from both the sides we get  
\begin{align}
&\sum_{\substack {i+j=n \\ i,j\geq 0}}\psi _i(\mu_j^{'}(a,b))=\sum_{\substack {i+j+k=n \\ i,j,k\geq 0}}\mu_i(\psi_j(a),\psi_k(b)),~~ a,b \in \mathfrak{g}.\\
& \sum_{\substack {i+j=n \\ i,j\geq 0}}\psi _i \circ T^{'}_j=\sum_{\substack {i+j=n \\ i,j\geq 0}} T_i \circ \psi _j.
\end{align}
Now putting $n=1$ in the above equation, we get 
\begin{align*}
&\mu^{'}_1(a,b)=\mu_1(a,b)+\mu (a,\psi_1 (b))+\mu (\psi_1(a),b)-\psi_1(\mu (a,b)) ,~~ a,b \in \mathfrak{g}\\
& T_1^{'}=T_1+T \circ \psi_1-\psi _1 \circ T.
\end{align*}
Therefore, we have
\[(\mu_1^{'},T_1^{'})-(\mu_1,T_1)=(\delta^1(\psi_1),-\phi^1(\psi_1))=d^1(\psi_1,0) \in C^{1}_{mRBLA}(\mathfrak{g},\mathfrak{g}).\]
Hence, from the above discussion, we have the following theorem.
\begin{theorem}
The infinitesimals of two equivalent one-parameter formal deformation of a modified Rota-Baxter Leibniz algebra $(\mathfrak{g} , \mu,T )$ of weight $\lambda$ are in the same cohomology class. 

\end{theorem} 
\begin{definition}
A modified Rota-Baxter Leibniz algebra $(\mathfrak{g}, \mu,T )$ of weight $\lambda$ is called rigid if every formal one-parameter deformation is trivial.
\end{definition}
\begin{theorem}
Let $(\mathfrak{g}, \mu,T)$ be a modified Rota-Baxter Leibniz algebra of weight $\lambda.$ Then $(\mathfrak{g}, \mu,T)$ is rigid if $H^2_{mRBLA}(\mathfrak{g},\mathfrak{g})=0.$
\end{theorem}
\begin{proof}
Let $(\mu_t, T_t)$ be a formal one-parameter deformation of the modified Rota-Baxter Leibniz algebra $(\mathfrak{g}, \mu,T)$. From Theorem \ref{infy-co}, $(\mu_1,T_1)$ is a $2$-cocycle 
and as  $H^2_{mRBLA}(\mathfrak{g},\mathfrak{g})=0$, thus, we have $(\mu_1,T_1)=d^1(\alpha,x)$, where $(\alpha,x) \in C^1_{LA}(\mathfrak{g},\mathfrak{g})\oplus Hom(\mathbb{K},\mathfrak{g}),~ \mathbb{K}$ being the ground field of the Leibniz algebra $(\mathfrak{g}, [~,~]).$ \\
Therefore, $\mu_1=\delta ^1 (\alpha)$ and $T_1=-\partial^0(x)-\phi^1(\alpha).$ Now, let $\psi_1=\alpha + \delta^0(x)$. Then $\mu_1= \delta^1(\psi_1)$ and $T_1=-\phi^1(\psi_1),~\text{as}~\phi ^1(\delta^0(x))=\partial^0(x)$.\\
Let $\psi_t=Id_{\mathfrak{g}}-\psi_1 t$, then $(\bar{\mu}_t,\bar{T}_t)$ is a formal one-parameter deformation, where
\[\bar{\mu}_t=\psi_t^{-1} \circ \mu_t \circ (\psi_t \circ \psi_t);~~ \bar{T}_t=\psi_t^{-1} \circ T_t \circ \psi_t. \]
It can be shown that $\bar{\mu_1}=0, \bar{T_1}=0$. Hence,
\begin{align*}
\bar{\mu}_t= \mu +\bar{\mu}_2t^2+\cdots ,  \\
\bar{T}_t= T +\bar{T}_2t^2+\cdots .
\end{align*}
 Again, one can show that $(\bar{\mu}_2, \bar{T}_2)$ is a $2$-cocycle. So by repeating the arguments, we can show that $(\mu_t, T_t)$ is equivalent to the trivial deformation. Hence, $(\mathfrak{g} , \mu,T)$ is rigid.
\end{proof}

\section{Abelian extension of modified Rota-Baxter Leibniz algebra of weight $\lambda$}\label{sec6}
Let $V$ be any vector space. We can always define a bilinear product on $V$ by $[a,b]=0$, i.e., $\mu (a,b)=0$ for all $a,b \in V$. If $T_V$ be a linear operator on $V$, then $(V,\mu,{T_V})$ is a modified Rota-Baxter operator of weight $\lambda$. Now we introduce the definition of the abelian extension of the modified Rota-Baxter Leibniz algebra of weight $\lambda$.
\begin{definition}
Let $(\mathfrak{g},[~,~],T)$ be a modified Rota-Baxter Leibniz algebra of weight $\lambda$ and $V$ be a vector space. Now a modified Rota-Baxter Leibniz algebra $(\hat{\mathfrak{g}},[~,~]_{\wedge},{\hat{T}})$ is called an abelian extension of $(\mathfrak{g},[~,~],T)$ by $V$ if there exists a short exact sequence of morphisms of modified Rota-Baxter Leibniz algebra of weight $\lambda$ 
\[
\begin{tikzcd}
0 \arrow[r] & (V,\mu,{T_V}) \arrow[r ,"i"] & (\hat{\mathfrak{g}},[~,~]_{\wedge},{\hat{T}}) \arrow[r,"p"] & (\mathfrak{g},[~,~],{T}) \arrow [r] & 0 .
\end{tikzcd} 
\]

\end{definition}
Therefore, for an abelian extension of the modified Rota-Baxter Leibniz algebra $(\mathfrak{g},[~,~],T)$ of weight $\lambda$, we have a commutative diagram :
\[
\begin{tikzcd}
0 \arrow[r] & V \arrow[r ,"i"] \arrow[d,"T_V"]& \hat{\mathfrak{g}} \arrow[d,"\hat{T}"]\arrow[r,"p"] & \mathfrak{g} \arrow [r]\arrow[d,"T"]  & 0 \\
 0 \arrow[r] & V \arrow[r ,"i"] & \hat{\mathfrak{g}} \arrow [r,"p"]  &  \mathfrak{g} \arrow [r]  & 0,
\end{tikzcd}
\]
where $\mu (a,b)=0$ for all $a,b \in V.$
\begin{definition}
Two abelian extensions $(\hat{\mathfrak{g}},[~,~]_{\wedge_1},{\hat{T_1}})$ and $(\hat{\mathfrak{g}},[~,~]_{\wedge_2},{\hat{T_2}})$ of a modified Rota-Baxter Leibniz algebra $(\mathfrak{g},[~,~],T)$ of weight $\lambda$ by a vector space $V$ is called an isomorphism if there exists an isomorphism of modified Rota-Baxter Leibniz algebra $\xi: (\hat{\mathfrak{g}},[~,~]_{\wedge_1},{\hat{T_1}}) \rightarrow  (\hat{\mathfrak{g}},[~,~]_{\wedge_2},{\hat{T_2}})$ such that the following diagram commutes
\[
\begin{tikzcd}
0 \arrow[r] & (V,\mu,{T_V}) \arrow[r ,"i"] \arrow[d,equal]& (\hat{\mathfrak{g}},[~,~]_{\wedge_1},{\hat{T_1}})\arrow[d,"\xi"]\arrow[r,"p"] & (\mathfrak{g},[~,~],T) \arrow [r]\arrow[d,equal]  & 0 \\
 0 \arrow[r] & (V,\mu,{T_V}) \arrow[r ,"i"] & (\hat{\mathfrak{g}},[~,~]_{\wedge_2},{\hat{T_2}}) \arrow [r,"p"]  &  (\mathfrak{g},[~,~],T) \arrow [r]  & 0.
\end{tikzcd}
\]

\end{definition}
\begin{definition}
Let $(\hat{\mathfrak{g}},[~,~]_{\wedge},{\hat{T}})$ be an abelian extension of $(\mathfrak{g},\mu,T)$ by a vector space $V$. Then a linear map $s: \mathfrak{g} \rightarrow \hat{\mathfrak{g}}$ is called a section if $p \circ s= Id_{\mathfrak{g}}.$
\end{definition}
 Here we will get a new representation from an abelian extension of a modified Rota-Baxter Leibniz algebra of weight $\lambda$. Suppose $(\hat{\mathfrak{g}},[~,~]_{\wedge},{\hat{T}})$ be an abelian extension of $(\mathfrak{g},[~,~],T)$ by a vector space  $V$ and $s: \mathfrak{g} \rightarrow \hat{\mathfrak{g}}$ be a section of it. We define $\bar{l}_V: \mathfrak{g} \otimes V \rightarrow V$ and $\bar{r}_V: V \otimes \mathfrak{g} \rightarrow V $ respectively by $\bar{l}_V(a,u)=[s(a),u]_{\wedge}$ and $\bar{r}_V(u,a)=[u,s(a)]_{\wedge}$, for all $a\in \mathfrak{g}, u \in V.$
\begin{theorem}

Let $(\hat{\mathfrak{g}},[~,~]_{\wedge},{\hat{T}})$ be an abelian extension of a modified Rota-Baxter Leibniz algebra $(\mathfrak{g},[~,~],T)$ of weight $\lambda$ by $(V,\mu,{T_V})$ and  $s: \mathfrak{g} \rightarrow \hat{\mathfrak{g}}$ be a section of it. Then, $(V,\bar{l}_V,\bar{r}_V,T_V)$ is a  representation of modified Rota-Baxter Leibniz algebra  $(\mathfrak{g},[~,~],T)$ of weight $\lambda.$

\end{theorem}
\begin{proof}
Observe that $s([a,b])-[s(a),s(b)]_{\wedge} \in V$, therefore $[s([a,b]),u]_{\wedge}=[[s(a),s(b)]_{\wedge},u]_{\wedge}$ for all $a,b \in \mathfrak{g}, u \in V$.\\
 Observe that
\begin{align*}
&\bar{l}_V(a,\bar{l}_V(b,u))-\bar{l}_V([a,b],u)-\bar{l}_V(b,\bar{l}_V(a,u)) \\
&= \bar{l}_V(a,[s(b),u]_{\wedge})-[s([a,b]),u]_{\wedge}-\bar{l}_V(b,[s(a),u]_{\wedge})\\
&=[s(a),[s(b),u]_{\wedge}]_{\wedge}-[s([a,b]),u]_{\wedge}-[s(b),[s(a),u]_{\wedge}]_{\wedge}\\
&=[s(a),[s(b),u]_{\wedge}]_{\wedge}-[[s(a),s(b)]_{\wedge},u]_{\wedge}-[s(b),[s(a),u]_{\wedge}]_{\wedge}\\
&=0.
\end{align*}
Similarly, we have the following equations 
\[\bar{l}_V(a,\bar{r}_V(u,b))=\bar{r}_V(\bar{l}_V(a,u),b)+\bar{r}_V(u,[a,b]),\]
\[ \bar{r}_V(u,[a,b]= \bar{r}_V(\bar{r}_V(u,a),y)+\bar{l}_V(a,\bar{r}_V(u,b)),\]
for all $a,b \in \mathfrak{g}$ and $u\in V.$ Thus, $(V,\bar{l}_V, \bar{r}_{V})$ is a representation of the Leibniz algebra $(\mathfrak{g}, [~,~]).$\\
Again, $s(T(a))-\hat{T}(s(a)) \in V$ holds, therefore, $[s(T(a)),u]=[\hat{T}(s(a)),u]$ for all $a\in \mathfrak{g}, u \in V$. Hence, we have
\begin{align*}
&\bar{l}_V(T(a),T_V(u))=[s(T(a)), T_V(u)]_{\wedge}=[\hat{T}(s(a)),\hat{T}(u)]_{\wedge}\\
&=\hat{T}([\hat{T}(s(a)),u]_{\wedge}+[s(a),\hat{T}(u)]_{\wedge})+\lambda [s(a),u]_{\wedge}\\
&=T_V([s(T(a)),u]_{\wedge}+[s(a),T_V(u)]_{\wedge})+\lambda [s(a),u]_{\wedge}\\
&=T_V(\bar{l}_V(T(a),u)+\bar{l}_V(a,T_V(u)))+\lambda \bar{l}_V(a,u).
\end{align*}

Hence, $\bar{l}_V(T(a),T_V(u))=T_V(\bar{l}_V(T(a),u)+\bar{l}_V(a,T_V(u)))+\lambda \bar{l}_V(a,u)$ for all $a,b \in \mathfrak{g}$, and $u\in V.$

Similarly, we have \[\bar{r}_V(T_V(u),T(a))=T_V(\bar{r}_V(T_V(u),a)+\bar{r}_V(u,T(a)))+\lambda \bar{r}_V(u,a),\]
for all $a \in \mathfrak{g}$ and $u \in V.$ Hence, $(V,\bar{l}_V,\bar{r}_V,T_V)$ is a representation of modified  Rota-Baxter Leibniz algebra  $(\mathfrak{g},[~,~],T)$ of weight $\lambda.$
\end{proof}
\begin{proposition}
Different choices of section of an abelian extension $(\hat{\mathfrak{g}},[~,~]_{\wedge},{\hat{T}})$ of a modified Rota-Baxter Leibniz algebra $(\mathfrak{g},[~,~],T)$ of weight $\lambda$ by $(V,\mu,{T_V})$ gives the same modified Rota-Baxter representation of weight $\lambda .$
\end{proposition}
\begin{proof}
Let $s_1$, $s_2$ be two distinct  sections of the abelian extension $(\hat{\mathfrak{g}},[~,~]_{\wedge},{\hat{T}})$. Now, define a map $\beta : \mathfrak{g} \rightarrow V$ by
\[\beta(a)=s_1(a)-s_2(a),~\mbox{for all }~ a \in \mathfrak{g}.\]
As $\mu (u,v)=0$ for all $ u,v \in V,$ we have
\begin{align*}
[s_1(a),u]_{\wedge}=[\beta(a)+s_2(a),u]_{\wedge}=[\beta(a),u]_{\wedge}+[s_2(a),u]_{\wedge}=[s_2(a),u]_{\wedge} .
\end{align*}
Similarly, it can be shown that, $[u,s_1(a)]_{\wedge}=[u,s_2(a)]_{\wedge}$ for all $a,b \in \mathfrak{g}, u \in V.$
Hence, two different sections give the same representation.
\end{proof}

\begin{proposition}
Let $(\hat{\mathfrak{g}},[~,~]_{\wedge},{\hat{T}})$ be an abelian extension of a modified Rota-Baxter Leibniz algebra $(\mathfrak{g},[~,~],T)$ of weight $\lambda$ by $(V,\mu,{T_V})$ and  $s: \mathfrak{g} \rightarrow \hat{\mathfrak{g}}$ be a section of it. Suppose the maps $\psi : \mathfrak{g}\otimes \mathfrak{g} \rightarrow V$ and $\chi : \mathfrak{g} \rightarrow V$ be defined by \begin{align*}
&\psi (a \otimes b)=[s(a),s(b)]_{\wedge}-s([a,b]),\\
&\chi (a)=\hat{T}(s(a))-s(T(a)), ~~ \mbox{for all}~ a,b \in \mathfrak{g}, \mbox{respcetively.}
\end{align*}
Then the cohomological class of $(\psi,\chi)$ is independent of the choice of sections.
\end{proposition}

\begin{proof}
Suppose $s_1$ and $s_2$ are two different abelian extension sections. Now define $\beta : \mathfrak{g} \rightarrow V$ by
\[\beta(a)=s_1(a)-s_2(a),~\mbox{for all }~ a \in \mathfrak{g}.\]
Note that $\mu (u,v)=0$ for all $ u,v \in V.$
Hence, for all $a,b \in \mathfrak{g}, u \in V,$ we have
\begin{align*}
\psi_1(a,b)
&=[s_1(a),s_1(b)]_{\wedge}-s_1([a,b])\\
&=[s_2(a)+\beta(a),s_2(b)+\beta(b)]_{\wedge}-(s_2([a,b])+\beta ([a,b]))\\
&=[s_2(a),s_2(b)]_{\wedge}-s_2([a,b])+[\beta(a),s_2(b)]_{\wedge}+[s_2(a),\beta (b)]_{\wedge}-\beta([a,b])\\
&=\psi_2(a,b)+\delta^1 (\beta)(a,b).
\end{align*}
Also,
\begin{align*}
\chi_1(a) &=\hat{T}(s_1(a))-s_1(T(a)) \\
&=\hat{T}(s_2(a)+\beta(a))-s_2(T(a))-\beta(T(a))\\
&=\chi_2(a)+T_V(\beta(a))-\beta (T(a))\\
 &=\chi_2(a)-\phi^1(\beta)(a).
\end{align*}
Therefore, $(\psi_1,\chi_1)-(\psi_2,\chi_2)=(\delta^1 (\beta), -\phi^1(\beta))=d^1(\beta)$. Hence, $(\psi_1,\chi_1)$ and $(\psi_2,\chi_2)$ are in the same cohomology class $H^2_{mRBLA}(\mathfrak{g},V)$.
\end{proof}
\begin{theorem}
Any two isomorphic abelian extensions of a modified Rota-Baxter Leibniz algebra $(\mathfrak{g},[~,~],T)$  of weight $\lambda$ by $(V,\mu,{T_V})$ give the same element in $H^2_{mRBLA}(\mathfrak{g}, V).$
\end{theorem}

\begin{proof}
Let $(\hat{\mathfrak{g}},[~,~]_{\wedge_1},{\hat{T_1}})$ and $(\hat{\mathfrak{g}},[~,~]_{\wedge_2},{\hat{T_2}})$ be two isomorphic abelian extensions of the modified Rota-Baxter Leibniz algebra $(\mathfrak{g},[~,~],T)$ of weight $\lambda$  by $(V,\mu,T)$. Let $s_1$ be a section of $(\hat{\mathfrak{g}},[~,~]_{\wedge_1},{\hat{T_1}})$. Now, $(\xi \circ s_1)$ is a section of $(\hat{\mathfrak{g}},[~,~]_{\wedge_2},{\hat{T_2}})$ as $p_2 \circ (s \circ s_1)= Id_{\mathfrak{g}}$ and $p_2 \circ s =p_1$, where $\xi$ is the map between the two abelian extensions.
Now define $s_2 :=\xi \circ s_1$. Since $\xi$ is a homomorphism of modified Rota-Baxter Leibniz algebras of weight $\lambda$ with  $\xi|_{V}=Id_V,~\mbox{and} ~\xi ([s_1(a),u]_{\wedge_1})=[s_2(a),u]_{\wedge_2}$ . Therefore,  $\xi|_{V}: V \rightarrow V$ is compatible with the induced representations.\\
Now, for all $a,b \in \mathfrak{g}$
\begin{align*}
&\psi_2(a \otimes b)=[s_2(a),s_2(b)]_{\wedge_2}-s_2([a,b])=[\xi(s_1(a)),\xi(s_1(b)])_{\wedge_2}-\xi(s_1([a,b]))\\
&=\xi ([s_1(a),s_1(b)]_{\wedge_1}-s_1([a,b]))=\xi (\psi_1 (a \otimes b))=\psi_1(a\otimes b),
\end{align*}
and
\begin{align*}
&\chi_2(a)= \hat{T_2}(s_2(a))-s_2(T(a))=\hat{T_2}(\xi(s_1(a)))-\xi(s_1(T(a)))\\
&=\xi (\hat{T_1 }(s_1(a))-s_1(T(a)))=\xi(\chi_1(a))=\chi_1(a).
\end{align*}
Hence, any two isomorphic abelian extensions give the same element in $H^2_{mRBLA}(\mathfrak{g}, V).$
\end{proof}

{\bf Acknowledgements:} 
The second/corresponding author is supported by the Core Research Grant (CRG) of Science and Engineering Research Board (SERB), Department of Science and Technology (DST), Govt. of India. (Grant Number- CRG/2022/005332). 
%The authors would like to thank the esteemed referee for carefully reading the paper and giving such constructive suggestions.

\renewcommand{\refname}{REFERENCES}


\begin{thebibliography}{99}



%\bibitem {Das} A. Das, S. K. Hazra, S. K. Mishra, Non-Abelian extension of Rota-Baxter Lie algebras and inducibility of automorphisms, arXiv:2204.01060v1.


\bibitem {Baxter}
G. Baxter, An analytic problem whose solution follows from a simple algebraic identity, \emph{Pacific J. Math.} 10(3) (1960), 731-742. 

  \bibitem{Cartier}
  P. Cartier, On the structure of free Baxter algebras, \emph{Advances in Math.} 9 (1972), 253–265. 
  
\bibitem{Das2} A. Das, A cohomological study of modified Rota-Baxter algebras, \emph{arXiv:2207.02273}.
   
\bibitem{das20}
A. Das, Deformations of associative Rota-Baxter operators, \emph{J. Algebra} 560 (2020) 144-180.   

\bibitem{Demir}
I. Demir, K. C. Misra, E. Stitzinger, On some structure of Leibniz algebras, \emph{Contemporary Mathematics}, 623 (2014),41-54.

\bibitem{FG}
K. Ebrahimi-Fard, L. Guo, Rota-Baxter algebras and dendriform algebras, \emph{J. Pure Appl. Algebra} 212 (2008), 320-339. 

\bibitem{fox}
T. F. Fox, An introduction to algebraic deformation theory, \emph{J. Pure Appl. Algebra}, 84(1) (1993) 17–41.

 \bibitem{G63}
 M. Gerstenhaber, The cohomology structure of an associative ring, \emph{Ann. Math.} 78, (1963)  267-288.

\bibitem{G64}
 M. Gerstenhaber, On the deformation of rings and algebras, \emph{Ann. Math.} (2) 79 (1964) 59-103.
 
 \bibitem{guo12}
L. Guo, An introduction to Rota-Baxter algebra, vol. 4, \emph{Surveys of Modern Mathematics, International Press}, Somerville, MA; Higher Education Press, Beijing, 2012.

\bibitem{GLS}
L. Guo, H. Lang, Y. Sheng, Integration and geometrization of Rota-Baxter Lie algebras, \emph{Adv. Math.} 387 (2021) Paper No. 107834.
 
 \bibitem {Guo} 
 S. Guo, Y. Qin, K. Wang and G. Zhou, Cohomology theory of Rota-Baxter Pre-Lie algebras of arbitrary weights, \emph{arXiv:2204.13518.}
 
 \bibitem{HLS}
 J. T. Hartwing, D. Larson and S. D. Silvestrov, Deformations of Lie algebras using $\sigma$-derivations, \emph{J. Algebra} 295 (2006) 314–361.
 
 \bibitem{Jiang}
 J. Jiang and Y. Sheng, Representations and cohomologies of relative Rota-Baxter Lie algebras and applications, \emph{J. Algebra} 602(2022), 637-670.
 
 \bibitem{LHB} 
 X. Li, D. Hou and C. Bai, Rota-Baxter operators on pre-Lie algebras, \emph{Journal of Nonlinear Mathematical Physics}, Volume 14, Number 2 (2007), 269–289.
 
 \bibitem {Loday}
 J. L. Loday, T. Pirashvili, Universal enveloping algebras of Leibniz algebras and (co)homology, \emph{Math. Ann.} 296 (1993) 139-158.
 
 \bibitem{MS23}
 A. Makhlouf, R. Saha, On Compatible Leibniz algebras, \emph{Journal of Algebra and Its Applications}, 2023, https://doi.org/10.1142/S0219498825501051. 
 
 \bibitem{Das3}
 S. K. Mishra, A. Das, S. K. Hazra,
Non-abelian extensions of Rota-Baxter Lie algebras and inducibility of automorphisms, \emph{Linear Algebra and its Applications}, Volume 669, 2023, Pages 147-174.
 
 \bibitem{MoS22}
B. Mondal, R. Saha, Cohomology, deformations, and extensions of Rota-Baxter Leibniz algebras, \emph{Communications in Mathematics}, Volume 30 (2022), Issue 2 (Special Issue : CIMPA School  "Non-associative algebras and Its Applications", Madagascar 2021).

\bibitem{MS22}
 G. Mukherjee and R. Saha, Equivariant one-parameter formal deformations of Hom-Leibniz algebras (2022),  \emph{Communications in Contemporary Mathematics}, Vol. 24, No. 03, 2050082 (2022).
 
 \bibitem{MS19}
 G. Mukherjee and R. Saha, Cup-product for equivariant Leibniz cohomology and Zinbiel algebra, \emph{Algebra Colloq.} 26(2) (2019) 271–284.
 
 \bibitem{NR66}
A. Nijenhuis and R. W. Richardson, Cohomology and deformations in graded Lie algebras, \emph{{Bull. Amer. Math. Soc.} 72 (1966) 1–29.}
 
 \bibitem {Rota}
  G. C. Rota, Baxter algebras and combinatorial identities I, II, \emph{Bull. Amer. Math. Soc.} 75 (1969), 325–329,  ibid 75 1969 330-334. 
   
%\bibitem {Jiang} J. Jiang, Y. Sheng, Cohomologies and deformations of modified $r$-matrices, arXiv:2206.00411.

\bibitem{saha}
R. Saha, Cup-product in Hom-Leibniz cohomology and Hom-Zinbiel algebras (2020), \emph{Communications in Algebra}, 48:10, 4224-4234.

\bibitem{sem}
M. Semonov-Tian-Shanskii, What is a classical R-matrix? \emph{Funct. Anal. Appl.} (1983) 259-272.

\bibitem{Tang}
 R. Tang, Y. Sheng and Y. Zhou, Deformations of relative Rota-Baxter operators on Leibniz algebras, \emph{Int. J. Geom. Methods Mod. Phys.} 17 (2020), no. 12, 2050174, 21 pp.
 
 \bibitem{uch}
K. Uchino, Quantum analogy of Poisson geometry, related dendriform algebras and Rota-Baxter operators, \emph{Lett. Math. Phys.} 85 (2008), no. 2-3, 91-109.

\bibitem{Wang} K. Wang and G. Zhou, Deformations and homotopy theory of Rota-Baxter algebras of any weight, Preprint, \emph{arXiv:2108.06744.} 

  \bibitem{Zhang1}
  X. Zhang, X. Gao and L. Guo, Modified Rota-Baxter Algebras, Shuffle Products and Hopf Algebras, \emph{Bull. Malays. Math. Sci. Soc.} 42 (2019), 3047-3072. 
  
   \bibitem{Zhang2}
    X. Zhang, X. Gao and  L. Guo, Free modified Rota-Baxter algebras and Hopf algebras, \emph{Int. Electron. J. Algebra} 25 (2019), 12-34.
   
\end{thebibliography}
\end{document}